\documentclass{amsart}%
\usepackage{amssymb}
\usepackage{amsfonts}
\usepackage{amsmath}
\usepackage{graphicx}%
\providecommand{\U}[1]{\protect\rule{.1in}{.1in}}
\newtheorem{theorem}{Theorem}
\theoremstyle{plain}

\newtheorem{corollary}{Corollary}

\newtheorem{definition}{Definition}
\newtheorem{example}{Example}

\newtheorem{lemma}{Lemma}

\numberwithin{equation}{section}
\begin{document}
\title{New Jensen-type inequalities}
\author{Constantin P. Niculescu}
\address{University of Craiova, Department of Mathematics, Street A. I. Cuza 13,
Craiova RO-200585, Romania}
\email{cniculescu47@yahoo.com}
\author{C\u{a}t\u{a}lin Irinel Spiridon}
\address{University of Craiova, Department of Mathematics, Street A. I. Cuza 13,
Craiova, RO-200585, Romania }
\email{catalin\_gsep@yahoo.com}
\date{March 14, 2012}

\begin{abstract}
We develop a new framework for the Jensen-type inequalities that allows us to
deal with functions not necessarily convex and Borel measures not necessarily positive.

\end{abstract}
\maketitle

It is well known the important role played by the classical inequality of
Jensen in probability theory, economics, statistical physics, information
theory etc. See \cite{NP2006} and \cite{PPT}. In recent years, a number of
authors have noticed the possibility to extend this inequality to the
framework of functions that are mixed convex (in the sense of the existence of
one inflection point). See \cite{Cz2006}, \cite{CP} and \cite{FN2007}. In all
these papers one assumes that both the function and the measure under
consideration verify certain conditions of symmetry. However the inequality of
Jensen is much more general as shows the following simple remark. Suppose that
$K$ is a convex subset of the Euclidean space $\mathbb{R}^{N}$ carrying a
Borel probability measure $\mu$. Then every $\mu$-integrable function
$f:K\rightarrow\mathbb{R}$ that admits a supporting hyperplane at the
barycenter of $\mu,$
\begin{equation}
b_{\mu}=\int_{K}xd\mu(x), \tag{$B$}\label{bar}%
\end{equation}
verifies the Jensen inequality%
\begin{equation}
f(b_{\mu})\leq\int_{K}f(x)d\mu(x). \tag{$J$}\label{J1}%
\end{equation}

Indeed, the existence of a supporting hyperplane at $b_{\mu}$ is equivalent to
the existence of an affine function $h(x)=\langle x,v\rangle+c$ such that%
\[
f(b_{\mu})=h(b_{\mu})\text{ and }f(x)\geq h(x)\text{ for all }x\in K.
\]
Then%
\[
f(b_{\mu})=h(b_{\mu})=h\left(  \int_{K}xd\mu(x)\right)  =\int_{K}%
h(x)d\mu(x)\leq\int_{K}f(x)d\mu(x).
\]

As is well known, the convexity assures the existence of a supporting
hyperplane at each interior point. See \cite{NP2006}, Theorem 3.7.1, p. 128.
This explains why Jensen's inequality works nicely in that context. The aim of
our paper is to extend the validity of Jensen's inequality outside mixed
convexity and also outside the framework of Borel probability measures.

In order to make our approach easily understandable we will restrict ourselves
to the case of functions of one real variable. However most of our results
extends easily to higher dimensions, by replacing the usual intervals by
$N$-dimensional intervals and symmetry with respect to a point by symmetry
with respect to a hyperplane. See Example 3 below.

We start with the following version of the Jensen inequality for mixed convex
functions that discards any assumption on the symmetry of the involved measure.

\begin{theorem}
\label{Thm1}Suppose that $f$ is a real-valued function defined on an interval
$[a,b]$ and $c$ is a point in $[a,\frac{a+b}{2}]$ such that:

$i)$ $f(c-x)+f(c+x)=2f(c)$ whenever $c\pm x\in\lbrack a,b];$

$ii)$ $f|_{[c,b]}$ is convex.

Then%
\[
f\left(  b_{\mu}\right)  \leq\int_{a}^{b}f(x)d\mu(x),
\]

for every Borel probability measure $\mu$ on $[a,b]$ whose barycenter lies in
the interval $[2c-a,b].$

The last inequality works in the reverse way when $f|_{[c,b]}$ is concave.
\end{theorem}

\begin{proof}
The case where $c=a$ is covered by the classical inequality of Jensen so we
may assume that $c\in(a,\frac{a+b}{2}).$ In this case the point $2c-a$ is
interior to $[a,b].$ By our hypotheses, the barycenter $b_{\mu}$ lies in the
interval $\left[  2c-a,b\right]  .$ If $b_{\mu}=b$, then $\mu=\delta_{b}$ and
the conclusion of Theorem \ref{Thm1} is clear. If $b_{\mu}$ is interior to
$[a,b],$ we will denote by $h$ the affine function joining the points
$(a,f(a))$ and $(2c-a,f(2c-a))$ and we will consider the function%
\begin{equation}
g(x)=\left\{
\begin{array}
[c]{cl}%
h(x) & \text{if }x\in\lbrack a,2c-a]\\
f(x) & \text{if }x\in\lbrack2c-a,b].
\end{array}
\right.  \tag{$1$}\label{funcg}%
\end{equation}
Clearly, $g$ is convex and this fact motivates the existence of a support line
$\ell$ of $g$ at $b_{\mu}.$ See \cite{NP2006}, Lemma 1.5.1, p. 30. Since
$h\geq f,$ then necessarily $\ell$ is a support line at $b_{\mu}$ also for
$f.$ By a remark above, this ends the proof.
\end{proof}

A useful consequence of Theorem \ref{Thm1} in the case of absolutely
continuous measures is as follows:

\begin{corollary}
\label{Cor1}Suppose that $f:[-b,b]\rightarrow\mathbb{R}$ is an odd function,
whose restriction to $[0,b]$ is convex and $p:[-b,b]\rightarrow\lbrack
0,\infty)$ is a nondecreasing function that does not vanish on $(-b/3,b].$
Then for every $a\in\lbrack-b/3,b),$%
\[
f\left(  \frac{1}{\int_{a}^{b}p(x)dx}\int_{a}^{b}xp(x)dx\right)  \leq\frac
{1}{\int_{a}^{b}p(x)dx}\int_{a}^{b}f(x)p(x)dx.
\]

\end{corollary}

\begin{proof}
The case where $a\geq0$ is covered by the classical inequality of Jensen.

If $a<0,$ then%
\begin{align*}
\int_{a}^{b}(x+a)p(x)dx  &  \geq\int_{a}^{-3a}(x+a)p(x)dx\\
&  =\int_{a}^{-a}(x+a)p(x)dx+\int_{-a}^{-3a}(x+a)p(x)dx\\
&  \geq\int_{a}^{-a}(x+a)p(x)dx+p(-a)\int_{-a}^{-3a}(x+a)dx\\
&  =\int_{a}^{-a}(x+a)p(x)dx-p(-a)\int_{a}^{-a}(x+a)dx\\
&  =\int_{a}^{-a}(x+a)\left(  p(x)-p(-a)\right)  dx\geq0,
\end{align*}
and thus Theorem 1 applies.
\end{proof}

An inspection of the argument of Corollary 1 shows that the monotonicity
hypothesis on $p$ can be relaxed by asking only the integrability of $p$ and
the fact that $p(x)\leq p(-a)\leq p(y)$ for all $x$ and $y$ with $x\leq-a\leq
y.$ However, simple examples show that the restriction $a\in\lbrack-b/3,b]$ in
Corollary 1 cannot be dropped.

Consider now the discrete version of Theorem 1.

\begin{corollary}
\label{CorD}Suppose that $f$ is a real-valued function defined on an interval
$I$ that contains the origin such that $f|_{I\cap\lbrack0,\infty)}$ is a
convex function and $f(-x)=-f(x)$ whenever $x$ and $-x$ belong to $I.$ Then
for every family of points $a_{1},...,a_{n}$ of $I$ and every family of
weights $p_{1},...,p_{n}\in\lbrack0,\infty)$ such that $\sum_{k=1}^{n}p_{k}=1$
and%
\[
\sum_{k=1}^{n}p_{k}a_{k}+\min\left\{  a_{1},...,a_{n}\right\}  \geq0,
\]
we have%
\[
f\left(  \sum_{k=1}^{n}p_{k}a_{k}\right)  \leq\sum_{k=1}^{n}p_{k}f(a_{k}).
\]

\end{corollary}

The conclusion of Corollary \ref{CorD} can be considerably improved when all
weights $p_{k}$ are equal.

\begin{corollary}
\label{CorStrongD}Suppose that $f$ is a real-valued function defined on an
interval $I$ that contains the origin. If $f|_{I\cap\lbrack0,\infty)}$ is a
convex function and $f(-x)=-f(x)$ whenever $x$ and $-x$ belong to $I,$ then
for every family of points $a_{1},...,a_{n}$ of $I$ such that%
\[
\text{ }\sum_{k=1}^{n}a_{k}+(n-2)\min\left\{  a_{1},...,a_{n}\right\}  \geq0
\]
we have%
\[
f\left(  \frac{1}{n}\sum_{k=1}^{n}a_{k}\right)  \leq\frac{1}{n}\sum_{k=1}%
^{n}f(a_{k}).
\]

\end{corollary}

\begin{proof}
It suffices to consider the case where $a_{1}\leq\cdots\leq a_{n}$ and
$a_{1}<0$. According to our hypothesis, $\sum_{k=1}^{n}a_{k}>0$ and
\[
\frac{1}{n-1}\sum_{k=2}^{n}a_{k}\geq-a_{1}\geq-a_{2}.
\]
By Corollary \ref{CorD},%
\[
f\left(  \frac{1}{n-1}\sum_{k=2}^{n}a_{k}\right)  \leq\frac{1}{n-1}\sum
_{k=2}^{n}f(a_{k}),
\]
and taking into account the function $g$ given by the formula (\ref{funcg}) we
infer that%
\begin{multline*}
f\left(  \frac{1}{n}\sum_{k=1}^{n}a_{k}\right)  \leq g\left(  \frac{1}{n}%
\sum_{k=1}^{n}a_{k}\right)  =g\left(  \frac{1}{n}\cdot a_{1}+\frac{n-1}%
{n}\cdot\frac{1}{n-1}\sum_{k=2}^{n}a_{k}\right) \\
\leq\frac{1}{n}g(a_{1})+\frac{n-1}{n}g(\frac{1}{n-1}\sum_{k=2}^{n}a_{k})\\
=\frac{1}{n}f(a_{1})+\frac{n-1}{n}f(\frac{1}{n-1}\sum_{k=2}^{n}a_{k})\leq
\frac{1}{n}\sum_{k=1}^{n}f(a_{k}).
\end{multline*}
The proof is done.
\end{proof}

A simple example illustrating Corollary \ref{CorStrongD} is
\[
\tan\left(  \frac{x+y+z}{3}\right)  \leq\frac{\tan x+\tan y+\tan z}{3},
\]
for every $x,y,z\in(-\pi/6,\pi/2)$ with $x+y+z+\min\left\{  x,y,z\right\}
\geq0.$

Starting with the pioneering work of J. F. Steffensen \cite{stef19}, a great
deal of research was done to extend the Jensen inequality outside the
framework of probability measures. An account on the present state of art can
be found in the monograph \cite{NP2006}, Sections 4.1 and 4.2. We will recall
here some basic facts for the convenience of the reader.

\begin{definition}
\label{Def1}\emph{A} Steffensen-Popoviciu measure \emph{is any real Borel
measure} $\mu$ \emph{on a compact convex set} $K$ \emph{such that}
$\mu(K)>0\ $\emph{and}%
\[
\int_{K}\,f(x)\,d\mu(x)\geq0\quad\text{\emph{for every continuous convex
function}\quad}f:K\rightarrow\mathbb{R}_{+}.
\]
\ 
\end{definition}

In the case of intervals, a complete characterization of this class of
measures is offered by the following result, independently due to T. Popoviciu
\cite{Po}, and A. M. Fink \cite{Fi}:

\begin{lemma}
\label{LemPF}Let $\mu$ be a real Borel measure on an interval $[a,b]$ with
$\mu([a,b])>0.$ Then $\mu$ is a Steffensen-Popoviciu measure if, and only if,
it verifies the following condition of endpoints positivity,%
\[
\int_{a}^{t}(t-x)\,d\mu(x)\geq0\ \text{and}\ \int_{t}^{b}(x-t)\,d\mu(x)\geq0
\]
$for\ every\ t\in\lbrack a,b].$
\end{lemma}

See \cite{NP2006}, p. 179, for details.$\allowbreak$

\begin{example}
\label{ex1}A discrete measure $\mu=\sum_{k\,=\,1}^{n}\,p_{k}\,\delta_{x_{k}}$
(supported by the points $x_{1}\leq...\leq x_{n}$) is a Steffensen-Popoviciu
measure if it verifies Steffensen's condition%
\begin{equation}
\sum_{k\,=\,1}^{n}\,p_{k}>0,\quad\text{and\quad}0\leq\sum_{k\,=\,1}^{m}%
\,p_{k}\leq\sum_{k\,=\,1}^{n}\,p_{k},\quad\text{for every }m\in\{1,...,n\}.
\tag{$2$}\label{dSt}%
\end{equation}

A concrete example is offered by the discrete measure $\frac{5}{9}%
\delta_{\frac{3a+b}{4}}-\frac{1}{9}\delta_{\frac{a+b}{2}}+\frac{5}{9}%
\delta_{\frac{a+3b}{4}}$.
\end{example}

\begin{example}
\label{ex2}According to Lemma \ref{LemPF}, $\left(  \ \frac{x^{2}}{a^{2}%
}-\frac{1}{6}\right)  dx$ is an example of absolutely continuous measure on
the interval $[-a,a]$ that is also a Steffensen-Popoviciu measure. As a
consequence, $\left(  \ \frac{x^{2}}{a^{2}}-\frac{1}{6}\right)  dxdy$ provides
an example of Steffensen-Popoviciu measure on any rectangle $[-a,b]\times
\lbrack c,d]$ with $0<a<b$ and $c<d.$ Indeed, if $h:[-a,b]\times\lbrack
c,d]\rightarrow\mathbb{R}$ is a nonnegative convex function, then
\[
x\rightarrow\left(  \int_{c}^{d}h(x,y)dy\right)
\]
is also a nonnegative convex function, whence
\[
\int_{-a}^{a}\int_{c}^{d}h(x,y)\left(  \frac{x^{2}}{a^{2}}-\frac{1}{6}\right)
dxdy\geq0.
\]
This yields%
\[
\int_{-a}^{b}\int_{c}^{d}h(x,y)\left(  \frac{x^{2}}{a^{2}}-\frac{1}{6}\right)
dxdy\geq0,
\]
and thus $\left(  \frac{x^{2}}{a^{2}}-\frac{1}{6}\right)  dxdy$ is a
Steffensen-Popoviciu measure on the interval $[-a,b]\times\lbrack c,d].$ Since
this class of measures is closed under addition, we infer that $\left(
\frac{x^{2}}{a^{2}}+\frac{y^{2}}{c^{2}}-\frac{1}{3}\right)  dxdy$ is a
Steffensen-Popoviciu measure on the interval $[-a,b]\times\lbrack-c,d]$,
whenever $0<a<b$ and $0<c<d.$
\end{example}

The Steffensen-Popoviciu measures provide the natural framework for the Jensen inequality:

\begin{theorem}
\label{Thm2}Suppose that $\mu$ is a Steffensen-Popoviciu measure on a compact
convex set $K$ $($part of a locally convex separated space $E$). Then $\mu$
admits a barycenter $b_{\mu}$ and for every continuous convex function $f$ on
$K,$%
\[
f(b_{\mu})\leq\frac{1}{\mu(K)}\,\int_{K\,}\,f(x)\,d\mu(x).
\]

\end{theorem}

For details, see \cite{NP2006}, Theorem 4.2.1, pp. 184-185. When $E$ is the
Euclidean space $\mathbb{R}^{N},$ the barycenter $b_{\mu}$ is given by the
formula $(B)$ above.

It is worth to mention that the argument of Theorem 1 remains valid in the
context of Steffensen-Popoviciu measures of total mass 1. Even more
importantly, it can be adapted to the case of functions of two or more variables.

\begin{example}
\label{ex3}Suppose $f:[-1,2]\times\lbrack-1,1]\rightarrow\mathbb{R}$ is a
function with the following two properties: $i)$ $f(-x,y)=-f(x,y)$ for every
$x\in\lbrack-1,1]$ and $y\in\lbrack-1,1]$; and $ii)$ $f|_{[0,2]\times
\lbrack-1,1]}$ is a convex function. According to Example \ref{ex2},
$(x^{2}-\frac{1}{6})dxdy$ is a Steffensen-Popoviciu measure on the interval
$[-1,2]\times\lbrack-1,1]$ $($of total mass $\frac{10}{3}$ and barycenter
$b_{\mu}=\left(  \frac{7}{5},0\right)  ).$ Since $\frac{7}{5}>1,$ an argument
similar to that used in the proof of Theorem 1 leads us to the following
Jensen-type inequality:%
\[
f\left(  \frac{7}{5},0\right)  \leq\frac{3}{10}\iint\nolimits_{[-1,2]\times
\lbrack-1,1]}f(x,y)(x^{2}-\frac{1}{6})dxdy.
\]

\end{example}

The above discussion leads us to the concept of almost convexity, whose
1-dimensional version is as follows:

\begin{definition}
\label{LAC}A real-valued function $f$ defined on an interval $I$ is left
almost convex if it is integrable and there is a pair of interior points $c<d$
in $I$ such that

$i)$ $f|_{[c,\infty)\cap I}$ is convex; and

$ii)$ $f\geq h$ on $(-\infty,c]\cap I,$ where $h$ is the affine function
joining the points $(c,f(c))$ and $(d,f(d)).$
\end{definition}

The concept of right almost convexity can be introduced in a similar way.

The following result extends Theorem 1 above.

\begin{theorem}
\label{Thm3}Suppose that $f:[a,b]\rightarrow\mathbb{R}$ is a left almost
convex function and $\mu$ is a Steffensen-Popoviciu measure on $[a,b]$ with
barycenter $b_{\mu}$. If $c<d$ are interior points to $[a,b]$ as in Definition
\ref{LAC} such that

$i)$ $b_{\mu}\geq c$ $;$ and

$ii)$ $\int_{a}^{d}\left(  \left(  f(x)-f(c)\right)  (d-c)-\left(
f(d)-f(c)\right)  (x-c)\right)  d\mu\geq0,$

\noindent then%
\[
f(b_{\mu})\leq\frac{1}{\mu([a,b])}\,\int_{a\,}^{b}\,f(x)\,d\mu(x).
\]

\end{theorem}

\begin{proof}
If $h$ is the affine function joining the points $(c,f(c))$ and $(d,f(d)),$
then the function%
\[
g(x)=\left\{
\begin{array}
[c]{cl}%
h(x) & \text{if }x\in\lbrack a,d]\\
f(x) & \text{if }x\in\lbrack d,b]
\end{array}
\right.
\]
is convex and $g(b_{\mu})=f(b_{\mu}).$ According to Theorem \ref{Thm2},%
\begin{align*}
f(b_{\mu})  &  =g(b_{\mu})\leq\frac{1}{\mu([a,b])}\int_{a\,}^{b}%
\,g(x)\,d\mu(x)\\
&  =\frac{1}{\mu([a,b])}\int_{a\,}^{d}\,h(x)\,d\mu(x)+\frac{1}{\mu([a,b])}%
\int_{d\,}^{b}\,f(x)\,d\mu(x)\\
&  \leq\frac{1}{\mu([a,b])}\int_{a\,}^{d}\,f(x)\,d\mu(x)+\frac{1}{\mu
([a,b])}\int_{d\,}^{b}\,f(x)\,d\mu(x)\\
&  =\frac{1}{\mu([a,b])}\int_{a}^{b}\,f(x)\,d\mu(x).
\end{align*}

\end{proof}

An illustration of Theorem \ref{Thm3} is offered by the following constrained
optimization problem: Find%
\[
M=\max_{(x,y,z)\in\Omega}\left[  \tan\left(  \frac{2x-y+3z}{4}\right)
-\frac{2\tan x-\tan y+3\tan z}{4}\right]  ,
\]
where $\Omega$ is the set of triplets $x\leq y\leq z$ in $[-\pi/3,\pi/3)$ such
that $2x-y+3z\geq0$ and%
\begin{equation}
\frac{\pi}{3\sqrt{3}}\left(  2\tan x-\tan y+3\tan z\right)  \geq2x-y+3z.
\tag{$3$}\label{3}%
\end{equation}
The answer is $M=0$. Indeed, according to (\ref{dSt}), the measure $\frac
{1}{2}\delta_{x}-\frac{1}{4}\delta_{y}+\frac{3}{4}\delta_{z}$ is
Steffensen-Popoviciu on the interval $[-\pi/3,\pi/3]$ and the constraint
(\ref{3}) coincides with the inequality $ii)$ in Theorem 3 (when applied to
the tangent function for $a=-\pi/3,$ $c=0$ and $b=d=\pi/3$).

Of course, the phenomenon of almost convexity is present also in higher
dimensions, at least for functions defined on \thinspace$N$-dimensional
intervals (or on other convex sets with a special geometry). It would be
interesting in that context to prove a characterization of the
Steffensen-Popoviciu measures (comparable to that offered by Lemma \ref{LemPF}).

\medskip

\noindent\textbf{Acknowledgement}. We acknowledge useful correspondence with
Florin Popovici. This work is supported by a grant of the Romanian National
Authority for Scientific Research, CNCS -- UEFISCDI, project number PN-II-ID-PCE-2011-3-0257.

\end{document}